\documentclass[12pt]{amsart}

\usepackage{amssymb}

\usepackage{}

\newtheorem{theorem}{Theorem}[section]
\newtheorem{lemma}[theorem]{Lemma}
\newtheorem*{theorem*}{Theorem}
\newtheorem*{proposition*}{Proposition}
\newtheorem*{metathm}{Meta-Theorem}
\newtheorem{corollary}[theorem]{Corollary}

\theoremstyle{definition}
\newtheorem{definition}[theorem]{Definition}

\theoremstyle{remark}
\newtheorem{remark}[theorem]{Remark}

\numberwithin{equation}{section}

\newcommand{\norm}[1]{\left\Vert#1\right\Vert}
\newcommand{\abs}[1]{\left\vert#1\right\vert}

\begin{document}

\title{}



\title[Characterization of Boundary Representations]{A matrix characterization of boundary representations of positive matrices in the Hardy space}
\author{John E. Herr}
\address[John E. Herr]{Department of Mathematics, Butler University, Indianapolis, Indiana 46208}
\email{jeherr@butler.edu}
\author{Palle E. T. Jorgensen}
\address[Palle E.T. Jorgensen]{Department of Mathematics, University of Iowa, Iowa City, IA 52242}
\email{palle-jorgensen@uiowa.edu}
\author{Eric S. Weber}
\address[Eric S. Weber]{Department of Mathematics, Iowa State University, 396 Carver Hall, Ames, IA 50011}
\email{esweber@iastate.edu}
\subjclass[2010]{Primary: 46E22, 30B30; Secondary 42B05, 28A25}
\date{\today}
\begin{abstract}
Spectral measures give rise to a natural harmonic analysis on the unit disc via a boundary representation of a positive matrix arising from a spectrum of the measure.  We consider in this paper the reverse: for a positive matrix in the Hardy space of the unit disc we consider which measures, if any, yield a boundary representation of the positive matrix.  We introduce a potential characterization of those measures via a matrix identity and show that the characterization holds in several important special cases.
\end{abstract}
\maketitle

\section{Introduction}

\subsection{The Szeg\H{o} Kernel}  
The classical Hardy space $H^2(\mathbb{D})$ consists of those holomorphic functions $f$ defined on $\mathbb{D}$ satisfying
\begin{equation}\label{Hardy}\lVert f\rVert^2_{H^2}:=\sup_{0<r<1}\int_{0}^{1}\lvert f(re^{2\pi ix})\rvert^2\,dx<\infty.\end{equation}
It is well-known that an equivalent description of $H^2(\mathbb{D})$ is as the space of holomorphic functions on $\mathbb{D}$ with square-summable coefficients:
$$H^2(\mathbb{D})=\left\{\sum_{n=0}^{\infty}c_nz^n~\middle|~\sum_{n=0}^{\infty}\lvert c_n\rvert^2<\infty\right\},$$
where the norm is then equivalently given by
$$\lVert f\rVert^2_{H^2}=\sum_{n=0}^{\infty}\lvert c_n\rvert^2.$$
In addition, for each $f\in H^2(\mathbb{D})$, there exists a (unique) function $f^\ast\in L^2(\mathbb{T})$, which we shall call the Lebesgue boundary function of $f$, such that
\begin{equation}\label{LebBoundary}\lim_{r\rightarrow1^{-}}\int_{0}^{1}\lvert f(re^{2\pi ix})-f^\ast(e^{2\pi ix})\rvert^2\,dx=0.\end{equation}
In fact, $\lim_{r\rightarrow1^-}f(re^{2\pi i x})=f^\ast(e^{2\pi ix})$ pointwise for almost every $x$. If $f(z)=\sum_{n=0}^{\infty}a_n z^n$ and $g(z)=\sum_{n=0}^{\infty}b_n z^n$ are two members of $H^2(\mathbb{D})$, the inner product of $f$ and $g$ in $H^2(\mathbb{D})$ can be described in two ways:
\begin{equation*}
{\langle f,g\rangle}_{H^2}=\sum_{n=0}^{\infty}a_n\overline{b_n}=\int_{0}^{1}f^\ast(e^{2\pi ix})\overline{g^\ast(e^{2\pi ix})}\,dx.\end{equation*}
Because the point-evaluation functionals on the Hardy space are bounded, the Hardy space is a reproducing kernel Hilbert space. Its kernel is the classical Szeg\H{o} kernel $k(w,z)=:k_w$, defined by
$$k_w(z):=\frac{1}{1-\overline{w}z}.$$
We then have
\begin{equation*}f(z)=\langle f,k_z\rangle_{H^2}=\int_{0}^{1}f^\ast(e^{2\pi ix})\overline{k_z^\ast(e^{2\pi ix})}\,dx\end{equation*}
for all $f\in H^2(\mathbb{D})$. In particular,
\begin{equation}\label{szegorep}k(w,z):=\int_{0}^{1}k_w^\ast(e^{2\pi ix})\overline{k_z^\ast(e^{2\pi ix})}\,dx.\end{equation} 
Equation $\eqref{szegorep}$ shows that the Szeg\H{o} kernel reproduces itself with respect to what is, by some definition, its boundary. The measure on the circle used to define $k_z^\ast$ in \eqref{LebBoundary} is Lebesgue measure, as is the measure in $\eqref{szegorep}$. The intent of this paper is to show that among the functions in the Hardy space, there are a host of other kernels that reproduce with respect to their boundaries. However, these boundary functions will not be taken with respect to Lebesgue measure, but with respect to some other  measure on the circle $\mathbb{T}$, and the integration of these boundary functions will also be done with respect to this measure. Of interest are two main questions: which positive matrices does the Hardy space contain that reproduce themselves by boundary functions with respect to a given measure, and with respect to which measures will a positive matrix reproduce itself by boundary functions?  In \cite{HJW16a}, we focused on the first question whereas in the present paper we focus on the second.

\subsection{Boundary Representations}  The boundary behavior of functions in the Hardy space arises in a number of contexts, such as the spectral theory of the shift operator \cite{Clark72}, de Branges-Rovnyak spaces \cite{dBR66a,Sar94}, and ``pseudo-continuable'' functions \cite{Aleks89a,Pol93}

The boundary behavior of kernels in reproducing kernel Hilbert spaces has been considered by others, see for example \cite{AD84a,AD85a,ADLW06a}.  In those papers, much of the theory is developed under the a priori assumption that a reproducing kernel Hilbert space ``sits isometrically'' within an $L^2$ space, by which the authors mean that the elements of the Hilbert space possesses an $L^2$ boundary as we define below and the two norms coincide.  However, the authors also begin with a measure on the boundary and define a reproducing kernel by integrating the Cauchy kernel against the measure \cite[Lemma 6.4]{AD84a}.  In the previous paper \cite{HJW16a}, we extend this result to construct many kernels which reproduce themselves with respect to a fixed singular measure using the Kaczmarz algorithm.  Other boundary representations were introduced in \cite{DJ11} to understand the nature of spectral measures.  In \cite{DJ11}, and spectral measure gives rise to a positive matrix on $\mathbb{D}$ which reproduces itself on the boundary with respect to the spectral measure.  As demonstrated in \cite{HJW16a}, however, the measure need not be spectral.

\begin{remark}In this paper, we will be dealing with measures $\mu$ on the unit circle. The unit circle $\mathbb{T}:=\{z\in\mathbb{C}:\lvert z\rvert=1\}$ and its topology shall be identified with $[0,1)$ via the relation $\xi=e^{2\pi ix}$ for $\xi\in\mathbb{T}$ and $x\in[0,1)$. We will regard the measures $\mu$ as being supported on $[0,1)$. A function $f(\xi)$ defined on $\mathbb{T}$ (for example, a boundary function) may be regarded as being in $L^2(\mu)$ if $f(e^{2\pi ix})\in L^2(\mu)$. For aesthetics, the inner product (norm) in $L^2(\mu)$ will be denoted $\langle\cdot,\cdot\rangle_{\mu}$ ($\lVert\cdot\rVert_{\mu}$) rather than $\langle\cdot,\cdot\rangle_{L^2(\mu)}$ ($\lVert\cdot\rVert_{L^2(\mu)}$). The subscript will be suppressed where context suffices.  A measure $\mu$ will be called \emph{singular} if it is a Borel measure that is singular with respect to Lebesgue measure.
\end{remark}

\begin{definition} \label{D:boundary} If $\mu$ is a finite Borel measure on $[0,1)$ and $f(z)$ is an analytic function on $\mathbb{D}$, we say that $f^{\star}\in L^2(\mu)$ is an $L^2(\mu)$-boundary function of $f$ if
\begin{equation*}\lim_{r\rightarrow1^-}\left\lVert f^{\star} (x)-f(re^{2\pi ix})\right\rVert_{\mu}=0.\end{equation*}
If a function possesses an $L^2(\mu)$-boundary function, then clearly that boundary function is unique. The $L^2(\mu)$-boundary function of a function $f:\mathbb{D}\rightarrow\mathbb{C}$ shall be denoted $f^{\star}_\mu$, but we omit the subscript when context precludes ambiguity.
\end{definition}

\begin{definition}A positive matrix (in the sense of E.~H.~Moore) on a domain $E$ is a function $K(z,w):E\times E\rightarrow\mathbb{C}$ such that for all finite sequences $\zeta_1,\zeta_2,\ldots,\zeta_n\in E$, the matrix
$$(K(\zeta_j,\zeta_i))_{ij}$$
is positive semidefinite. 
\end{definition}

Our interest is in positive matrices on $E=\mathbb{D}$, and more specifically those residing in $H^2(\mathbb{D})$. Recall that the classical Hardy space is a reproducing kernel Hilbert space. We therefore desire to find subspaces of the Hardy space that not only are Hilbert spaces with respect to the $L^2(\mu)$-boundary norm, but are in fact reproducing kernel Hilbert spaces with respect to this norm. The classical Moore-Aronszajn Theorem connects positive matrices to reproducing kernel Hilbert spaces \cite{Aron50a}:

\begin{theorem*}[Moore-Aronszajn]To every positive matrix $K(w,z)$ on a domain $E$ there corresponds one and only one class of functions on $E$ with a uniquely determined quadratic form in it, forming a Hilbert space and admitting $K(w,z)$ as a reproducing kernel. This class of functions is generated by all functions of the form $\sum_{k=1}^{n}\xi_k K(w_{k},z)$, with norm given by
$$\left\lVert\sum_{k=1}^{n}\xi_k K(w_k,z)\right\rVert^2=\sum_{i,j=1}^{n}K(z_j,z_i)\overline{\xi_i}\xi_j.$$
\end{theorem*}

Conversely, every reproducing kernel of a Hilbert space of functions on a common domain is a positive matrix. Let us then define two sets of interest:

\begin{definition}Let $\mu$ be a Borel measure on $[0,1)$. We define $\mathcal{K}(\mu)$ to be the set of positive matrices $K$ on $\mathbb{D}$ such that for each fixed $z\in\mathbb{D}$, $K(w,\cdot)$ possesses an $L^2(\mu)$-boundary $K^{\star}(w,\cdot)$, and $K(w,z)$ reproduces itself with respect to integration of these $L^2(\mu)$-boundaries, i.e.
\begin{equation} \label{Eq:reproduce}
K(w,z)=\int_{0}^{1}K^{\star}(w,x)\overline{K^{\star}(z,x)}\,d\mu(x)
\end{equation}
for all $z,w\in\mathbb{D}$.\end{definition}

\begin{definition}Let $K$ be a positive matrix on $\mathbb{D}$. We define $\mathcal{M}(K)$ to be the set of nonnegative Borel measures $\mu$ on $[0,1)$ such that for each fixed $z\in\mathbb{D}$, $K(w, \cdot)$ possesses an $L^2(\mu)$-boundary $K^{\star}(w, \cdot)$, and $K(w,z)$ reproduces itself with respect to integration of these $L^2(\mu)$-boundaries as in Equation (\ref{Eq:reproduce}).\end{definition}

\begin{definition}A sequence $\{x_n\}_{n=0}^{\infty}$ in a Hilbert space $\mathbb{H}$ is called a frame \cite{DS52} if there exist positive constants $A$ and $B$ such that
\begin{equation}\label{framecond}A\lVert\phi\rVert^2\leq\sum_{n=0}^{\infty}\lvert\langle \phi,x_n\rangle\rvert^2\leq B\lVert\phi\rVert^2\end{equation}
for all $\phi\in\mathbb{H}$. If $\{x_n\}_{n=0}^{\infty}$ satisfies (possibly only) the right-hand inequality in \eqref{framecond}, it is called a Bessel sequence. If $A=B$, the frame is called tight, and if $A=B=1$, it is called a Parseval frame.\end{definition}


The quaternary Cantor measure $\mu_4$ is the restriction of the $\frac{1}{2}$-dimensional Hausdorff measure to the quaternary Cantor set. Likewise, the ternary Cantor measure $\mu_3$ is the restriction of the $\frac{\ln(2)}{\ln(3)}$-dimensional Hausdorff measure to the ternary Cantor set. In \cite{JP98}, Jorgensen and Pedersen showed that the quaternary Cantor measure is spectral. That is, there exists a set $\Gamma \subset\mathbb{Z}$ such that the set of complex exponentials $\left\{e^{2\pi in x}\right\}_{\lambda \in \Gamma}$ is an orthonormal basis of $L^2(\mu_4)$. From this, Dutkay and Jorgensen \cite{DJ11} constructed a positive matrix $G_\Gamma$ inside $H^2$ that reproduces itself both in $H^2$ and with respect to $L^2(\mu_4)$-boundary integration. Thus $G_\Gamma\in\mathcal{K}(\mu_4)$.

In \cite{JP98}, it was also shown that $\mu_3$ is not spectral. Thus, it is not possible to construct a positive matrix for $\mu_3$ in the same way as for $\mu_{4}$. However, it is sufficient for $\mu_{3}$ to possess an exponential frame:
\begin{proposition*}
If there exists a sequence $\{n_{j}\}_{j=0}^{\infty}$ of nonnegative integers such that $\{ e^{2 \pi i n_{j} x} : j \geq 0\} \subset L^2(\mu)$ is a frame, then $\mathcal{K}(\mu)$ is nonempty.
\end{proposition*}
We proved a generalization of this result in \cite{HJW16a}.  It is still unknown whether $\mu_3$ possesses an exponential frame, which motivates our interest in understanding the boundary representations of positive matrices.

\subsection{Kernels from a Coefficient Matrix}

Let $C = (c_{mn})$ be a bi-infinite matrix, where $m,n \geq 0$.  We consider the formal power series
\begin{equation} \label{Eq:KC}
K_{C}(w, z) = \sum_{n} \sum_{m} c_{mn} \overline{w}^{m} z^{n}.
\end{equation}
We shall assume that $c_{nm} = \overline{c}_{mn}$; we shall make additional assumptions on $C$ as needed.  For example, if we assume that $\{ c_{mn} \}$ is a bounded sequence, then the formal power series $K_{C}$ converges absolutely on $\mathbb{D} \times \mathbb{D}$, and thus $K_{C}$ is holomorphic on $\mathbb{D}$ in $z$ and antiholomorphic on $\mathbb{D}$ in $w$.  For the remainder of the paper, we shall assume the coefficient sequence is bounded.

Moreover, we wish $K_{C}$ to be a positive matrix on $\mathbb{D} \times \mathbb{D}$, so we assume that $C$ has this property.  When the matrix $C$ defines a bounded linear operator on $\ell^{2}(\mathbb{N}_{0})$, then $K_{C}$ is a positive matrix if and only if $C$ is a positive operator.  Indeed, for $z \in \mathbb{D}$, we denote by $\vec{z}$ the element of $\ell^2(\mathbb{N}_{0})$ where $( \vec{z} )_{n} = (z^{n})_{n}$.  Then, for $z_{1}, \dots , z_{N} \in \mathbb{D}$ and $\xi_{1}, \dots , \xi_{N} \in \mathbb{C}$,
\[ \sum_{j=1}^{N} \sum_{k=1}^{N} \xi_{j} \overline{\xi}_{k} K_{C}(z_{k}, z_{j}) = \sum_{j=1}^{N} \sum_{k=1}^{N} \xi_{j} \overline{\xi}_{k} \langle C \vec{z}_{j}, \vec{z}_k \rangle_{\ell^{2}} 
= \left\langle C \sum_{j=1}^{N} \xi_{j} \vec{z}_{j} , \sum_{k=1}^{N} \xi_{k} \vec{z}_{k} \right\rangle_{\ell^{2}} \]
which is nonnegative if and only if $C$ is a positive self-adjoint operator on $\ell^{2}(\mathbb{N}_{0})$.  Assuming that $C$ is a bounded linear operator on $\ell^{2}(\mathbb{N}_{0})$ has the additional virtue that for every $w \in \mathbb{D}$, $K_{C}(w, \cdot) \in H^2(\mathbb{D})$, since the coefficient sequence $(\overline{ C \vec{ w } } )_{m}$ is square-summable.  We have established the following:

\begin{lemma} \label{L:KCbounded}
If $C = (c_{mn})_{mn}$ is a bounded, positive, self-adjoint operator on $\ell^{2}(\mathbb{N}_{0})$, then the kernel $K_{C}$ as given in Equation (\ref{Eq:KC}) is a positive matrix such that for each $w \in \mathbb{D}$, $K_{C}(w, \cdot) \in H^{2}(\mathbb{D})$.  Moreover, for $w, z \in \mathbb{D}$,
\begin{equation} \label{Eq:Ceq}
K_{C}(w,z) = \langle C \vec{z}, \vec{w} \rangle_{\ell^2}.
\end{equation}
\end{lemma}

For a given $C$ which defines a positive matrix as in Equation (\ref{Eq:KC}), we wish to determine which Borel measures on $\mathbb{T}$, if any, are in $\mathcal{M}(K_{C})$.  We shall approach the question via the following meta-theorem:

\begin{metathm}
A measure $\mu$ is in $\mathcal{M}(K_{C})$ if and only if the matrix equation $C = CMC$ is satisfied, where the matrix $M = ( \hat{\mu}(n-m) )_{mn}$.
\end{metathm}

We describe this as a meta-theorem for several reasons.  First, even if $C$ is a bounded operator on $\ell^{2}(\mathbb{N}_{0})$, the expression $CMC$ may not be well-defined.  Indeed, \emph{a priori} this product is only defined when $C$ and $M$ are bounded operators on $\ell^{2}(\mathbb{N}_{0})$; we may have only one or neither of these matrices with that property.  Second, the matrix equality does not \emph{a priori} assure that the kernel functions $K_{C}(w, \cdot)$ have $\mu$-boundaries.  Our goal in the present paper is to establish the meta-theorem for two special cases: i) for diagonal matrices $C$, and ii) for $C$ and $\mu$ for which $M$ which are bounded operators on $\ell^{2}(\mathbb{N}_{0})$.  We have a description of which $\mu$ has the property that $M$ is bounded \cite{Cas00a,DHSW11}, see also \cite{Lai12}:
\begin{lemma}
The matrix $M = ( \hat{\mu}(n-m) )_{mn}$ is a bounded operator on $\ell^2(\mathbb{N}_{0})$ if and only if $\mu << \lambda$ and the Radon-Nikodym derivative $\dfrac{d \mu}{d \lambda} \in L^{\infty}(\lambda)$.
\end{lemma}

\section{Lebesgue Measure: Kernels in $H^{2}(\mathbb{D})$ with Equal Norms}

We assume that the coefficient matrix $C$ defines a bounded, positive, self-adjoint operator on $\ell^{2}(\mathbb{N}_{0})$ and consider initially the special case of Lebesgue measure.

\begin{theorem} \label{Th:Lebesgue}
Suppose $C = (c_{mn})$ defines a bounded, positive, self-adjoint operator on $\ell^{2}(\mathbb{N}_{0})$.  The following are equivalent:
\begin{enumerate}
\item $\lambda \in \mathcal{M}(K_{C})$;
\item the coefficient matrix $C$ is a projection;
\item the norm induced by $K_{C}$ is equal to the Hardy space norm in the following sense: for all $\xi_{1}, \dots, \xi_{N} \in \mathbb{C}$ and $w_{1}, \dots , w_{N} \in \mathbb{D}$,
\[ \left\| \sum_{j=1}^{N} \xi_{j} K_{C}(w_{j}, \cdot) \right\|_{K_{C}} = \left\| \sum_{j=1}^{N} \xi_{j} K_{C}(w_{j}, \cdot) \right\|_{H^{2}}; \]
\item there exists a subspace $M$ of the Hardy space such that the Parseval frame $g_{n} = P_{M}z^{n}$ is such that $c_{mn} = \langle g_{n}, g_{m} \rangle$;
\item there exists a subspace $M$ of the Hardy space such that the projection of the Szeg\"o kernel onto $M$ is $K_{C}$.
\end{enumerate}
\end{theorem}

\begin{remark}  The equivalence of 1 and 2 would follow immediately from our meta-theorem.  We will establish the meta-theorem for absolutely continuous measures with bounded Radon-Nikodym derivative in the next section--we present here a proof that uses only the equality of norms.
\end{remark}

\begin{proof}
$(1 \Leftrightarrow 3)$  If $\lambda \in \mathcal{M}(K_{C})$, then
\begin{align*}
\left\| \sum_{j=1}^{N} \xi_{j} K_{C}(w_{j}, \cdot) \right\|_{H^{2}}^{2} &= \int_{0}^{1}  \left(\sum_{j=1}^{N} \xi_{j} K_{C}^{\star}(w_{j}, \cdot) \right) \left( \overline{\sum_{k=1}^{N} \xi_{k} K_{C}^{\star}(w_{k}, \cdot) } \right) d \lambda \\
&= \sum_{j=1}^{N} \sum_{k=1}^{N} \xi_{j} \overline{\xi}_{k} \int_{0}^{1} K_{C}^{\star}(w_{j}, \cdot) \overline{ K_{C}^{\star}(w_{k}, \cdot) } d \lambda \\
&= \sum_{j=1}^{N} \sum_{k=1}^{N} \xi_{j} \overline{\xi}_{k} K_{C}(w_{j}, w_{k}) \\
&= \left\| \sum_{j=1}^{N} \xi_{j} K_{C}(w_{j}, \cdot) \right\|_{K_{C}}^{2}.
\end{align*}
Conversely, if the norms are equal, we have by the polarization identity
\begin{align*}
K_{C}(w,z) &= \langle K_{C}(w, \cdot), K_{C}(z, \cdot) \rangle_{K_{C}} \\
&= \langle K_{C}(w, \cdot), K_{C}(z, \cdot) \rangle_{H^{2}} \\
&= \int_{0}^{1} K_{C}^{\star}(w, \cdot) \overline{K_{C}^{\star}(z, \cdot)} d \lambda.
\end{align*}


$(2 \Leftrightarrow 3)$  Consider the following calculations:  
\begin{align}
\left\| \sum_{j=1}^{N} \xi_{j} K_{C}(w_{j}, \cdot) \right\|_{K_{C}}^{2} &= \sum_{j=1}^{N} \sum_{k=1}^{N} \xi_{j} \overline{\xi}_{k} K_{C}(w_{j}, w_{k}) \notag \\ 
&= \sum_{j=1}^{N} \sum_{k=1}^{N} \xi_{j} \overline{\xi}_{k} \left\langle C \vec{w}_{k}, \vec{w}_{j} \right\rangle_{\ell^2} \notag \\
&= \left\langle C \left( \sum_{k=1}^{N} \overline{\xi}_{k} \vec{w}_{k} \right), \sum_{j=1}^{N} \overline{\xi}_{j} \vec{w}_{j} \right\rangle_{\ell^2} \label{Eq:KCnorm}
\end{align}
and
\begin{align}
\left\| \sum_{j=1}^{N} \xi_{j} K_{C}(w_{j}, \cdot) \right\|_{H^{2}}^{2} &= \left\| \sum_{n=0}^{\infty} \left( \sum_{m=0}^{\infty} c_{mn} \sum_{j=1}^{N} \xi_{j} \overline{w}_{j}^{m} \right) z^{n} \right\|_{H^{2}}^{2} \notag \\
&= \left\| \sum_{n=0}^{\infty} \overline{ \left( \sum_{m=0}^{\infty} c_{nm} \sum_{j=1}^{N} \overline{\xi}_{j} w_{j}^{m} \right) } z^{n} \right\|_{H^{2}}^{2} \notag \\
&= \left\langle C \left( \sum_{j=1}^{N} \overline{\xi}_{j} \vec{w}_{j} \right), C \left( \sum_{k=1}^{N} \overline{\xi}_{k} \vec{w}_{k} \right) \right\rangle_{\ell^2}.  \label{Eq:H2norm}
\end{align}
It follows that if $C$ is projection, then the inner-products in Equations (\ref{Eq:KCnorm}) and (\ref{Eq:H2norm}) are equal.  Conversely, if the norms are equal, then by the polarization identity, we have that $C^2 = C$; since $C$ is assumed self-adjoint, $C$ is a projection.

$(3 \Leftrightarrow 5)$  If the norms are equal, then the RKHS generated by $K_{C}$ is a closed subspace $M$ of $H^2(\mathbb{D})$ (with equal norm), and hence the projection $P_{M}$ of the Szeg\"o kernel is the reproducing kernel for $M$, as is $K_{C}$.  Conversely, if $K_{C}$ is the projection of the Szeg\"o kernel onto $M$, then the norms are equal.

$(2 \Leftrightarrow 4)$  If $C$ is a projection, then we can define $\Phi : H^2(\mathbb{D}) \to H^2(\mathbb{D})$ by 
\[ \Phi f (z) = \sum_{m} (C \vec{f})_{m} z^{m}, \quad \text{where} \quad f(z) = \sum_{n} f_{n} z^{n} \quad \text{and} \quad \vec{f} = (f_{n})_{n}. \]
It is readily verified that $\Phi$ is a projection on $H^2(\mathbb{D})$.  We have
\[ c_{mn} = \langle \Phi z^{n}, z^{m} \rangle = \langle \Phi z^{n}, \Phi z^{m} \rangle. \]
Conversely, if $g_{n} = P_{M} z^{n}$, then $\{g_{n}\}$ is a Parseval frame, thus its Grammian matrix is a projection.
\end{proof}


\section{Diagonal Coefficient Matrices}

We consider the special case of when $C$ is a diagonal matrix.  Let $\Gamma \subset \mathbb{N}_{0}$ and consider $K_{\Gamma}(w,z) = \sum_{\gamma \in \Gamma} (z \overline{w})^{\gamma}$.  We will see that either a) there are many absolutely continuous measures in $\mathcal{M}(K_{\Gamma})$, or b) only Lebesgue measure is in $\mathcal{M}(K_{\Gamma})$.  The determining factor of which possibility occurs is the difference set of $\Gamma$.

\subsection{The Kernels $K_{4}$ and $K_{3}$}
Two specific kernels that fall into this category that we wish to understand are the kernels $K_{3}$ and $K_{4}$.  Recall that a spectrum for $\mu_{4}$ is
\[ \Gamma_{4} := \left\{ \sum_{j=0}^{N} l_{j} 4^{j} | l_{j} \in \{0,1\} \right\} = \{ 0 , 1, 4, 5, 16, 17, 20, 21, \dots \}. \]
Then
\[ K_{4} (w,z) := \sum_{ \gamma \in \Gamma_{4} } (\overline{w} z)^{\gamma} = \prod_{j=0}^{\infty} \left( 1 + (\overline{w}z  )^{4^j} \right). \]
An introduction to $K_{4}$ appears in \cite{DJ11}, where it is shown that $\mu_{4} \in \mathcal{M}(K_{4})$.  We show in Corollary \ref{C:K4} below that there are many (absolutely continuous) measures in $\mathcal{M}(K_{4})$.

We also consider the kernel $K_{3}$, defined analogously to $K_{4}$:
\[ K_{3}(w,z) := \prod_{j=0}^{\infty} \left( 1 + (\overline{w}z  )^{3^j} \right) = \sum_{n \in \Gamma_{3}} (\overline{w} z )^{n}, \]
where 
\[  \Gamma_{3} = \left\{ \sum_{j=0}^{N} l_{j} 3^{j} | l_{j} \in \{0, 1\} \right\} = \{0, 1, 3, 4, 9, 10, 12, 13, \dots \}. \]
We shall show in Corollary \ref{C:K3} below that $\mathcal{M}(K_{3})$ contains only Lebesgue measure.

Note that $K_{\Gamma}$ corresponds to the diagonal coefficient matrix $C$ with $c_{mm} = 1$ if and only if $m \in \Gamma$, and $c_{mm} = 0$ otherwise.  Therefore $C$ is a projection, and hence as a consequence of Theorem \ref{Th:Lebesgue}, we have:

\begin{corollary} \label{C:lambda}
For any $\Gamma \subset \mathbb{N}_{0}$, $\lambda \in \mathcal{M}(K_{\Gamma})$.
\end{corollary}

We shall also consider diagonal matrices $C$ which are not projections--in fact we can consider diagonal matrices $C$ which are not bounded operators on $\ell^2(\mathbb{N}_{0})$.  For example, the Bergmann kernel is given by
\[ K_{B}(w,z) = \sum_{n=0}^{\infty} (n+1) (\overline{w} z)^{n}. \]
We shall show in Corollary \ref{C:empty} below that there are no representing measures for $C$ which have distinct nonzero diagonal entries.

\subsection{The Meta-Theorem for Diagonal Coefficient Matrices}

For two matrices $A = (a_{mn})$ and $B = (b_{mn})$, we say that $AB$ is \emph{defined in the matrix sense} if for every $m,n \in \mathbb{N}_{0}$, the sum $\sum_{k=0}^{\infty} = a_{mk} b_{kn}$ converges.  Note that this holds if $A,B$ are bounded operators on $\ell^{2}(\mathbb{N}_{0})$. We say $ABC$ is defined in the matrix sense if $AB$, $BC$, $(AB)C$, and $A(BC)$ are defined in the matrix sense and $(AB)C = A(BC)$.

\begin{theorem} \label{Th:diag}
Suppose $C$ is diagonal matrix such that $c_{nn} \geq 0$ and for every $0 < r < 1$, $\sum c_{nn} r^{n} < +\infty$. Let $\mu$ be a Borel probability measure on $[0,1)$ with $M=(\widehat{\mu}(n-m))_{mn}$. Then the following hold:
\begin{enumerate}
\item $\sum_{n=0}^{\infty}c_{nn}\overline{w}^n e^{2\pi i n x}$ converges in $L^2(\mu)$.
\item $CMC$ is defined in the matrix sense.
\item $K_{C}(w,z)$ reproduces itself with respect to $L^2(\mu)$ boundaries if and only if the equation $C=CMC$ holds.
\end{enumerate}
\end{theorem}

\begin{proof}For the first part, we have for any $\abs{w}<1$,
\[ \sum_{n=0}^{\infty}\norm{c_{nn}\overline{w}^{n}e^{2\pi in x}}_{\mu} = \sum_{n=0}^{\infty}\abs{c_{nn}\overline{w}^{n}} <\infty. \]
Thus, $\sum_{n=0}^{\infty} c_{nn}\overline{w}^{n}e^{2\pi in x}$ is absolutely summable in $L^2(\mu)$ and thus converges in $L^2(\mu)$.

For the second part, observe that 
\begin{align*} 
(CM)_{mn}&=\sum_{k=0}^{\infty}c_{mk}\hat{\mu}(n - k) = c_{mm}\hat{\mu}(n - m) \\
(MC)_{mn}&=\sum_{k=0}^{\infty}\hat{\mu}(k - m)c_{kn} = \hat{\mu}(n-m) c_{nn}.
\end{align*}
Likewise,
\begin{align*}
(C(MC))_{mn}&=\sum_{k=0}^{\infty}c_{mk} \hat{\mu}(n - k) c_{nn}= c_{mm} c_{nn} \hat{\mu}(n - m) = c_{mm}M_{mn}c_{nn}\\
((CM)C)_{mn}&=\sum_{k=0}^{\infty}c_{mm} \hat{\mu}(k - m) c_{kn} = c_{mm} \hat{\mu}(n - m) c_{nn} = c_{mm}M_{mn}c_{nn}.
\end{align*}
This shows that $CM$, $MC$, $C(MC)$, and $(CM)C$ are defined in the matrix sense, and that $C(MC)=(CM)C$.

Now, suppose $K_{C}(w,z)$ reproduces itself with respect to $L^2(\mu)$ boundary. The first part, together with Abel summability, shows that $K_{C}^\star(w,x)=\sum_{m=0}^{\infty}c_{mm}\overline{w}^{n} e^{2\pi im x}$.

By continuity of the inner product in $L^2(\mu)$, we have
\begin{align*}\int_{0}^{1}K_{C}^\star(w,x) \overline{K_{C}^\star(z,x)}\,d\mu(x)&=\sum_{n = 0}^{\infty}\overline{c_{nn}} \left( \int_{0}^{1}K_{C}^\star(w,x) e^{-2\pi in x}\,d\mu(x) \right) z^{n} \\
&=\sum_{n = 0}^{\infty}\sum_{m=0}^{\infty}c_{nn}c_{mm} \left(\int_{0}^{1} e^{2 \pi i m x } e^{-2 \pi i n x} \,d\mu(x) \right) \overline{w}^{m}z^{n}\\
&=\sum_{n = 0}^{\infty}\sum_{m=0}^{\infty}c_{nn}c_{mm} \hat{\mu}(n-m) \overline{w}^{m}z^{n}\\
&=\sum_{n=0}^{\infty}\sum_{m=0}^{\infty} c_{mm} M_{mn} c_{nn} \overline{w}^{m}z^{n}\\
&=\sum_{n=0}^{\infty}\sum_{m=0}^{\infty} (CMC)_{mn} \overline{w}^{m}z^{n}
\end{align*}
Therefore, Equation (\ref{Eq:reproduce}) holds if and only if
\[ \sum_{n=0}^{\infty}\sum_{m=0}^{\infty} c_{mn} \overline{w}^{m}z^{n} = \sum_{n=0}^{\infty}\sum_{m=0}^{\infty} (CMC)_{mn} \overline{w}^{m}z^{n} \]
holds, which by uniqueness of Taylor coefficients, holds if and only if $C = CMC$.
\end{proof}

\begin{corollary} \label{C:empty}
Suppose $C$ is a diagonal matrix which satisfies the hypotheses of Theorem  \ref{Th:diag} and which has two distinct nonzero diagonal entries.  Then $\mathcal{M}(K_{C}) = \emptyset$.
\end{corollary}

\begin{proof}
Suppose $\mu \in \mathcal{M}(K_{C})$.  Then we must have $c_{mm} = (C M C)_{mm} = c_{mm} M_{mm} c_{mm} = \| \mu \| c_{mm}^2$ for all $m \in \mathbb{N}_{0}$.  Thus, for any nonzero diagonal entry $c_{mm} = \| \mu \|$.
\end{proof}

\begin{definition}
For a set $A \subset \mathbb{R}$, the difference set is 
\[ \mathcal{D}(A) = \{ x - y \ | \ x,y \in A \}. \]
\end{definition}

\begin{corollary} \label{C:FC}
If $\Gamma \subset \mathbb{N}_{0}$ and $\mu$ is a probability measure, $\mu \in \mathcal{M}(K_{\Gamma})$ if and only if
\begin{equation} \label{Eq:Fourier}
\hat{\mu} (n) = 0
\end{equation}
for all $n \in \mathcal{D}(\Gamma) \setminus \{0\}$.
\end{corollary}

\begin{proof}
We verify $C=CMC$ holds if and only if Equation (\ref{Eq:Fourier}) holds.  We have $c_{mn} = 1$ if $m=n \in \Gamma$ and $0$ otherwise.  Thus by the calculation in Theorem \ref{Th:diag},
\[ (C M C)_{mn} = M_{mn} \]
if $m,n \in \Gamma$ and
\[ (C M C)_{mn} = 0 \]
otherwise.  Thus, $C=CMC$ holds if and only if $M_{mn} = 0$ whenever $m, n \in \Gamma$ with $m \neq n$.  The result now follows since
\[ M_{mn} = \hat{ \mu }(n-m). \]
\end{proof}

\subsection{The Kernels $K_{3}$ and $K_{4}$, continued}

\begin{corollary} \label{C:K3}
Suppose $\Gamma \subset \mathbb{N}_{0}$ is such that $\mathcal{D}(\Gamma) = \mathbb{Z}$.  Then $\mathcal{M}(K_{\Gamma}) = \{ \lambda\}$.  In particular, $\mathcal{M}(K_{3}) = \{ \lambda \}$.
\end{corollary}

\begin{proof}
By Corollary \ref{C:lambda}, we have $\lambda \in \mathcal{M}(K_{\Gamma})$.  Now, suppose $\mu \in \mathcal{M}(K_{\Gamma})$.  We must have that the matrix equation $C = CMC$ is satisfied.  Thus, for $m \neq n \in \Gamma$, we have $0 = c_{mn} = c_{mm} M_{mn} c_{nn} = \hat{\mu}(n - m)$.  Since the difference set of $\Gamma$ is $\mathbb{Z}$, it follows that $\hat{\mu}(k) = 0$ for $k \in \mathbb{Z} \setminus \{0\}$, whence $\mu$ must be Lebesgue measure.  The claim for $K_{3}$ is a consequence of Lemma \ref{L:Gamma3}.
\end{proof}

\begin{lemma} \label{L:Gamma3}
The difference set $\mathcal{D}(\Gamma_{3}) = \mathbb{Z}$.
\end{lemma}

\begin{proof}
We prove that $\mathcal{D}(\Gamma_{3})$ is invariant under the iterated functions $\varphi_{0}(x) = 3x$, $\varphi_{1}(x) = 3x + 1$, and $\varphi_{2}(x) = 3x -1$.  Indeed, suppose that $n \in \mathcal{D}(\Gamma_{3})$, then $n = \eta_{1} - \eta_{2}$ for $\eta_{k} \in \Gamma_{3}$.  Since $\Gamma_{3}$ is invariant under $\varphi_{0}$ and $\varphi_{1}$, we have
\begin{align*}
 \varphi_{0}(n) &= \varphi_{0}(\eta_{1}) - \varphi_{0}(\eta_{2}) \in \mathcal{D}(\Gamma_{3}) \\
 \varphi_{1}(n) &= \varphi_{1}(\eta_{1}) - \varphi_{0}(\eta_{2}) \in \mathcal{D}(\Gamma_{3}) \\
 \varphi_{2}(n) &= \varphi_{0}(\eta_{1}) - \varphi_{1}(\eta_{2}) \in \mathcal{D}(\Gamma_{3}).
\end{align*}
Clearly $\{-1, 0,1\} \subset \mathcal{D}(\Gamma_{3})$, so since it is invariant under $\varphi_{0}, \varphi_{1}$, $\varphi_{2}$, our claim is established.
\end{proof}

\begin{remark}
A consequence of Corollary \ref{C:K3} is that $\Gamma_{3}$ is \emph{not} a spectrum of any measure.  Likewise, any $\Gamma \subsetneq \mathbb{Z}$ whose difference set is $\mathbb{Z}$ is not a spectrum of any measure.
\end{remark}

\begin{corollary} \label{C:K4}
If $\Gamma \subset \mathbb{N}_{0}$ is such that $\mathcal{D}(\Gamma) \neq \mathbb{Z}$, then there exist absolutely continuous measures in $\mathcal{M}(K_{\Gamma})$.  In particular, $\mathcal{M}(K_{4})$ contains many absolutely continuous measures.
\end{corollary}

\begin{proof}
We define $\mu$ by its Radon-Nikodym derivative: choose
\[ \dfrac{ d \mu }{ d \lambda }(\theta) = 1 + \sum_{n \notin \mathcal{D}(\Gamma)} b_{n} \cos(2 \pi n \theta) \]
subject to the constraint that $\sum_{n \notin \mathcal{D}(\Gamma)} | b_{n} | < 1$.  It follows that $\mu$ is a probability measure such that 
\[ \widehat{\dfrac{ d \mu }{ d \lambda } }(n) = 0 \]
for $n \in \mathcal{D}(\Gamma)\setminus\{0\}$, and so satisfies Corollary \ref{C:FC}.

Now, we claim that $\mathcal{D}(\Gamma_{4}) \neq \mathbb{Z}$.  Indeed, if we define
\[ \Gamma_{4}^{\prime} = \left\{ \sum_{j = 0}^{N} l_j 4^{j} | l_{j} \in \{0, 2\} \right\}, \]
then we claim that $\Gamma_{4}^{\prime} \cap \mathcal{D}(\Gamma_{4}) = \{ 0 \}$.  To establish this, suppose we have
\[ \sum_{j=0}^{N_{1}} l_{j} 4^j = \sum_{j=0}^{N_{2}} p_{j} 4^{j} - \sum_{j=0}^{N_{3}} q_{j} 4^j \]
with $l_{j} \in \{0,2\}$ and $p_{j},q_{j} \in \{ 0, 1 \}$.  We may assume $N_{1} = N_{2} = N_{3}$ by padding with $0$'s if necessary.  Thus, we have
\[ \sum_{j=0}^{N} p_{j} 4^{j} = \sum_{j=0}^{N} l_{j} 4^j + \sum_{j=0}^{N} q_{j} 4^j = \sum_{j=0}^{N} (l_{j} + q_{j}) 4^{j}, \]
where $l_{j} + q_{j} \in \{0, 1, 2, 3\}$.  Since the base $4$ expansion is unique, we must have that $p_{j} = l_{j} + q_{j}$ for all $j$, which can only occur when $l_{j} = 0$ for all $j$.
\end{proof}

\section{Absolutely Continuous Measures}

We proceed now to prove the Meta-Theorem in the case that the Grammian matrix $M$ is a bounded operator on $\ell^{2}(\mathbb{N}_{0})$.  As mentioned previously, this occurs when the measure $\mu$ is absolutely continuous with bounded Radon-Nikodym derivative.

\begin{theorem} \label{Th:AC}
Suppose $C$ is a bounded, positive, self-adjoint operator on $\ell^{2}(\mathbb{N})$, $\mu << \lambda$, and $\dfrac{d \mu}{d \lambda} \in L^{\infty}(\mathbb{T})$.  Then $\mu \in \mathcal{M}(K_{C})$ if and only if $C = C M C$, where $M$ is the Grammian matrix of $\{ e_{n} \}_{n=0}^{\infty} \subset L^2(\mu)$, i.e. $M_{mn} = \hat{\mu}(n-m)$.
\end{theorem}

\begin{proof}
Since $\dfrac{ d \mu }{ d \lambda } \in L^{\infty}(\mathbb{T})$, the sequence $\{ e_{n} \}_{n=0}^{\infty} \subset L^2(\mu)$ is a Bessel sequence.  Thus, $M$ is a bounded operator on $\ell^{2}(\mathbb{N})$, and the matrix product $CMC$ is defined.  Moreover, for every $ w \in \mathbb{D} $, we have that since 
\[ \left( C \vec{w} \right)_{n} \in \ell^2(\mathbb{N}) \]
the series
\[ \sum_{n=0}^{\infty} \left( \sum_{m=0}^{\infty} c_{mn} \overline{w}^{m} \right) e^{2 \pi i n t} \]
converges in $L^2(\mu)$.  Thus, we have that for every $w \in \mathbb{D}$, the $L^2(\mu)$ boundary is given by
\[ K_{C}^{\star}(w, t ) = \sum_{n = 0}^{\infty} \left( \sum_{m=0}^{\infty} c_{mn} \overline{w}^{m} \right) e^{2 \pi i n t} \]
by Abel summation.  We calculate
\begin{align}
\int_{0}^{1} K_{C}^{\star}(w, t) \overline{K_{C}^{\star}(z,t)} d \mu(t) &= \sum_{n=0}^{\infty} \left( \sum_{m=0}^{\infty} c_{mn} \overline{w}^{m} \right) \int_{0}^{1} e^{2 \pi i n t} \overline{K_{C}^{\star}(z,t)} d \mu(t) \notag \\
&= \sum_{n=0}^{\infty} \left( \sum_{m=0}^{\infty} c_{mn} \overline{w}^{m} \right) \sum_{k=0}^{\infty} \left( \sum_{l=0}^{\infty} \overline{c}_{lk} z^{l} \right) \int_{0}^{1} e^{2 \pi i n t} e^{- 2 \pi i k t} d \mu(t) \notag \\
&=\sum_{n=0}^{\infty} \left( \sum_{m=0}^{\infty} c_{mn} \overline{w}^{m} \right)  \sum_{k=0}^{\infty} \left( \sum_{l=0}^{\infty} c_{k l} z^{l} \right) M_{n k}. \label{Eq:CM1}
\end{align}
We have by the Cauchy-Schwarz inequality:
\begin{align*}
\sum_{l} \sum_{k} | c_{k l} M_{n k} z^{l} | &\leq \sum_{l} | z^{l} | \sqrt{ \sum_{k} | c_{k l} |^2 } \sqrt{ \sum_{k} | M_{n k} |^2 } \\
&\leq \|C \| \|M \| \sum_{l} | z^{l} | \\
&< \infty. 
\end{align*}
Therefore, 
\begin{equation} \label{Eq:CM2} 
(\ref{Eq:CM1}) = \sum_{n=0}^{\infty} \left( \sum_{m=0}^{\infty} c_{mn} \overline{w}^{m} \right)  \sum_{l=0}^{\infty} \left( MC \right)_{n l} z^{l}.
\end{equation}
Again by the Cauchy-Schwarz inequality:
\begin{align*}
\sum_{m} \sum_{n} \sum_{l} | c_{mn} \left( MC \right)_{n l} z^{l} \overline{w}^n | &\leq \sum_{m} \sum_{l} | z^{l} \overline{w}^n | \sqrt{ \sum_{n} | c_{mn}|^2} \sqrt{ \sum_{n} | \left( MC \right)_{n l}} |^2 \\
&\leq \| C \| \| MC \| \sum_{m} \sum_{l} | z^{l} \overline{w}^n | \\
&< \infty.
\end{align*}
Whence
\begin{equation} \label{Eq:CMC}
(\ref{Eq:CM2}) = \sum_{l=0}^{\infty} \sum_{m=0}^{\infty} (CMC)_{m l} \overline{w}^{m} z^{l}.
\end{equation}

Consequently,  Equation (\ref{Eq:reproduce}) if and only if
\begin{equation} \label{Eq:C=CMC}
\sum_{n = 0}^{\infty} \sum_{m = 0}^{\infty} c_{mn} \overline{w}^{m} z^{n} = \sum_{n=0}^{\infty} \sum_{m=0}^{\infty} (CMC)_{m n} \overline{w}^{m} z^{n}.
\end{equation}
Equation (\ref{Eq:C=CMC}) holds if and only if $C=CMC$ by the uniqueness of Taylor series coefficients.
\end{proof}

\section{Preservation of Norms of Subspaces of $L^2(\lambda)$}

The proofs of Theorems \ref{Th:diag} and \ref{Th:AC} suggest that the property that $K_{C}$ reproduces itself with respect to some $\mu$ on the boundary is related to the following question: given a closed subspace $V \subset L^2(\lambda)$, for which measures $\nu$ does the following norm preservation identity hold for all $f \in V$:
\[ \int |f|^2 d \lambda = \int |f|^2 d \nu? \]
Of course, this is ill-defined, because for $f \in L^2(\lambda)$, the question of whether $f \in L^2(\nu)$ and subsequently norm equality may depend on the representative.  However, this ambiguity can be made precise using the boundary behavior of kernels as in Theorems \ref{Th:norm1} and \ref{Th:norm2}. 

\begin{definition}
Suppose $V \subset L^{2}_{+}(\lambda)$ is a closed subspace and let $\widetilde{V} \subset H^{2}(\mathbb{D})$ be the space consisting of all functions whose $L^2(\lambda)$ boundaries are in $V$.  We say the measure $\mu$ preserves the norm of $V$ if for every $f \in V$, with $f(x) = \sum_{n = 0}^{\infty} a_{n} e^{2 \pi i n x}$, the corresponding function $F(z) = \sum_{n=0}^{\infty} a_{n} z^n$ has a $L^{2}(\mu)$ boundary $F^{\star}$ and
\[ \| f \|_{\lambda} = \| F^{\star} \|_{\mu}. \]
\end{definition}

\begin{lemma} \label{L:norm1}
Suppose $\Gamma \subset \mathbb{N}_{0}$ and $\mu$ satisfies Equation \ref{Eq:Fourier}, then for every $f$ in the subspace generated by $K_{\Gamma}$, $f$ possesses a $L^2(\mu)$-boundary, and the norm of the $L^2(\mu)$-boundary agrees with the $H^2(\mathbb{D})$ norm of $f$.
\end{lemma}

\begin{proof}
Consider $\Gamma_{1} \subset \Gamma$ of finite cardinality, and $f(z) = \sum_{\gamma \in \Gamma_{1}} a_{\gamma} z^{\gamma}$.  Let $f^{\star}( e^{2 \pi i \theta} ) = \sum_{\gamma \in \Gamma_{1}} a_{\gamma} e^{2 \pi i \gamma \theta}$.  We claim that $f^{\star}$ is the $L^2(\mu)$-boundary of $f$ and that
\[ \| f^{\star} \|_{\mu} = \| f \|_{H^2}. \]
Indeed, we have that the $L^2(\mu)$-boundary of the function $z^{\gamma}$ is $e^{2 \pi i \gamma \theta}$ by uniform convergence, and thus by linearity $f^{\star}$ is the $L^2(\mu)$-boundary of $f$.  Moreover, 
\begin{align}
\| f^{\star} \|_{\mu}^{2} &= \int_{0}^{1} \left( \sum_{\gamma \in \Gamma_{1}} a_{\gamma} e^{2 \pi i \gamma \theta} \right) \left( \sum_{\gamma^{\prime} \in \Gamma_{1}} \overline{a_{\gamma^{\prime}}} e^{-2 \pi i \gamma^{\prime} \theta} \right) d \mu \notag \\
&= \sum_{\gamma \in \Gamma_{1}} \sum_{\gamma^{\prime} \in \Gamma_{1}} a_{\gamma} \overline{a_{\gamma^{\prime}}} \int_{0}^{1} e^{2 \pi i (\gamma - \gamma^{\prime})} d \mu \notag \\
&= \sum_{\gamma \in \Gamma_{1}} | a_{\gamma} |^2 \notag \\
&= \| f \|_{H^2}^{2}. \label{Eq:equalnorms}
\end{align}

Now, for $f(z) = \sum_{\gamma \in \Gamma} a_{\gamma} z^{\gamma}$, the series
\[ f^{\star}(e^{2 \pi i \theta}) = \sum_{\gamma \in \Gamma} a_{\gamma} e^{2 \pi i \gamma \theta} \]
converges in $L^2(\mu)$, and $f^{\star}$ is the $L^2(\mu)$-boundary for $f$ by Abel summability.  The equality of norms follows from taking limits in Equation (\ref{Eq:equalnorms}).
\end{proof}

\begin{theorem} \label{Th:norm1}
Let $\Gamma \subset \mathbb{N}_{0}$ and let $V$ be the closed span of $\{ e^{2 \pi i \gamma \theta} \}_{\gamma \in \Gamma}$ in $L^2(\lambda)$.  The measure $\mu$ preserves the norm of $V$ if and only if $C = CMC$, where $C$ is the diagonal matrix $c_{nn} = 1$ if $n \in \Gamma$ and $0$ otherwise, and $M = (\hat{\mu}(n-m))_{mn}$.
\end{theorem}

\begin{proof}
We have that $K_{C} (=K_{\Gamma})$ is the reproducing kernel for the space $\widetilde{V}$, and thus, if $\mu$ preserves the norm of $V$, we have by the polarization identity that $\mu \in \mathcal{M}(K_{C})$.  By Theorem \ref{Th:diag} we must have $C = CMC$.

Conversely, if $C = CMC$, then $\mu \in \mathcal{M}(K_{C})$, and thus $\mu$ preserves the norms of finite linear combinations $\sum_{j=1}^{N} b_{j} K_{C}^{\star}(w_{j}, \cdot)$:
\begin{align*}
\| \sum_{j=1}^{N} b_{j} K_{C}^{\star}(w_{j}, \cdot) \|_{\mu}^{2} &= \sum_{j = 1}^{N} \sum_{k = 1}^{N} b_{j} \overline{b_{k}} \int_{0}^{1} K_{C}^{\star}(w_{j}, \cdot ) \overline{K_{C}^{\star}(w_{k}, \cdot )} d \mu \\
&= \sum_{j = 1}^{N} \sum_{k = 1}^{N} b_{j} \overline{b_{k}} K_{C}(w_{j}, w_{k}) \\
&= \| \sum_{j=1}^{N} b_{j} K_{C}(w_{j}, \cdot) \|_{H^{2}}^{2} \\
&= \| \sum_{j=1}^{N} b_{j} K_{C}^{\star}(w_{j}, \cdot) \|_{\lambda}^{2}.
\end{align*}

We see by the proof of Corollary \ref{C:FC} that if $C = CMC$, then $\mu$ satisfies Equation (\ref{Eq:Fourier}).  By Lemma \ref{L:norm1}, every element of the space spanned by $K_{C}$ possesses an $L^2(\mu)$ boundary, and by density, $\mu$ then preserves the norms of all elements of $V$. 
\end{proof}

\begin{lemma} \label{L:norm2}
Suppose $C$ is a projection on $\ell^{2}(\mathbb{N}_{0})$ and $\mu << \lambda$ is such that $\dfrac{d \mu}{d \lambda} \in L^{\infty}(\lambda)$.  Let $N = ( \hat{\mu}(m-n))_{mn}$.  If $C = CNC$, then for every sequence $(a_{n})_{n}$ in the range of $C$  
\[ \| \sum a_{n} e_{n} \|_{\mu} = \| \sum a_{n} e_{n} \|_{\lambda}. \]
\end{lemma}

Note that $N = M^{T}$ in our previous notation.

\begin{proof}
Our hypotheses yield that the series $\sum_{n} a_{n} e^{2 \pi i n \theta}$ converges in $L^2(\mu)$.  We have
\begin{align*}
\| \sum a_{n} e^{2 \pi i n \theta} \|^{2}_{\mu} &= \sum_{n,m} a_{n} \overline{a_{m}} \int_{0}^{1} e^{2 \pi i (n-m) \theta} d \mu(\theta) \\
&= \langle N (a_{n})_{n}, (a_{n})_{n} \rangle \\
&= \langle C N C (a_{n})_{n}, (a_{n})_{n} \rangle \\
&= \langle C (a_{n})_{n}, (a_{n})_{n} \rangle \\
&= \| \sum a_{n} e^{2 \pi i n \theta} \|^{2}_{\lambda}.
\end{align*}

\end{proof}

\begin{lemma} \label{L:CT}
Suppose $V$ is a subspace of $L^2_{+}(\lambda)$.  Let $C$ be the projection on $\ell^{2}(\mathbb{N}_{0})$ such that $f = \sum_{n} a_{n} e_{n} \in V$ if and only if $(a_{n})_{n}$ is in the range of $C$.  Then the reproducing kernel of $\widetilde{V}$ is $K_{C^{T}}$.
\end{lemma}

\begin{proof}
First note that for $w \in \mathbb{D}$,
\[ K_{C^T} (w, z) = \sum_{n} \sum_{m} (C^{T})_{mn} \overline{w}^{m} z^{n} = \sum_{n} \left( C \vec{\overline{w}} \right)_{n} z^{n}  \]
is such that the coefficients are in the range of $C$.  Thus, $K_{C}(w, \cdot) \in \widetilde{V}$.  Now, for $f(z) = \sum_{n} a_{n} z^{n} \in \widetilde{V}$, we have that
\begin{align*}
\langle f , K_{C^{T}} \rangle_{H^{2}} &= \langle (a_{n})_{n} , C \vec{ \overline{w}} \rangle_{\ell^2} \\
&= \langle (a_{n})_{n} , \vec{ \overline{w}} \rangle_{\ell^2} \\
&= \sum_{n} a_{n} w^{n} \\
&= f(w),
\end{align*}
and thus $K_{C^{T}}$ is the kernel as claimed.
\end{proof}

\begin{theorem} \label{Th:norm2}
Suppose $V$ is a subspace of $L^{2}_{+}(\lambda)$, and let $\mu < < \lambda$ with $\dfrac{d \mu}{d \lambda} \in L^{\infty}(\lambda)$.  Then $\mu$ preserves the norm of $V$ if and only if $C^{T} = C^{T}MC^{T}$, where $C$ is the projection on $\ell^{2}(\mathbb{N}_{0})$ with the property that $K_{C^{T}}$ is the reproducing kernel of $\widetilde{V} \subset H^{2}(\mathbb{D})$ and $M = (\hat{\mu}(n-m))_{mn}$.
\end{theorem}

\begin{proof}
For $w_{1},w_{2} \in \mathbb{D}$, $K_{C^{T}}(w_{j}, z) = \sum_{n} ( C \vec{\overline{w}}_{j} )_{n} z^{n}$.  By our assumptions, the $L^2(\mu)$ boundary of $K_{C^{T}}(w_{j}, \cdot)$ is $\sum_{n} ( C \vec{\overline{w}}_{j} )_{n} e^{2 \pi i n \theta}$.

Suppose that $\mu$ preserves the norm of $V$.  We have by the Polarization Identity:
\begin{align*}
\int_{0}^{1} K_{C^{T}}^{\star}(w_{1}, \theta) \overline{K_{C^{T}}^{\star}(w_{2}, \theta)} d \mu(\theta) &= \left\langle K_{C^{T}}^{\star} (w_{1}, \cdot) , K_{C^{T}}^{\star}(w_{2}, \cdot) \right\rangle_{\mu} \\
&= \left\langle \sum_{n} ( C \vec{\overline{w}}_{1} )_{n} e^{2 \pi i n \theta} , \sum_{n} ( C \vec{\overline{w}}_{2} )_{n} e^{2 \pi i n \theta} \right\rangle_{\lambda} \\
&= \langle K_{C^{T}} (w_{1}, \cdot) , K_{C}(w_{2}, \cdot) \rangle_{H^{2}} \\
&= K_{C^{T}}(w_{1}, w_{2}).
\end{align*}
Therefore, $\mu \in \mathcal{M}(K_{C^{T}})$, and hence by Theorem \ref{Th:AC}, $C^{T} = C^{T}MC^{T}$.

Conversely, if $C^{T} = C^{T}MC^{T}$, then $C = CNC$.  Therefore, for every finite linear combination $\sum_{j=1}^{N} \xi_{j} K^{\star}_{C^{T}} (w_{j} , \cdot) \in V$, we have by Lemma \ref{L:norm2} that
\[ \| \sum_{j=1}^{N} \xi_{j} K^{\star}_{C^{T}} (w_{j} , \cdot) \|_{\mu} = \| \sum_{j=1}^{N} \xi_{j} K^{\star}_{C^{T}} (w_{j} , \cdot) \|_{\lambda}. \]
By density, $\mu$ preserves the norm of $V$.
\end{proof}

\bibliographystyle{amsplain}

\providecommand{\bysame}{\leavevmode\hbox to3em{\hrulefill}\thinspace}
\providecommand{\MR}{\relax\ifhmode\unskip\space\fi MR }
\providecommand{\MRhref}[2]{%
  \href{http://www.ams.org/mathscinet-getitem?mr=#1}{#2}
}
\providecommand{\href}[2]{#2}

\end{document}